\documentclass[a4paper,10pt]{article}
\usepackage[utf8x]{inputenc}
\pdfoutput=1

\usepackage{times}
\usepackage{geometry}
\usepackage[T1]{fontenc}
\usepackage{color}
\usepackage{amsmath}
\usepackage{amssymb}
\usepackage{array}
\usepackage{amsthm}
\usepackage{graphicx}
\usepackage{hyperref}
\usepackage{setspace}
\usepackage{stmaryrd}

\theoremstyle{plain}
\newtheorem{thm}{Theorem}[section]
\newtheorem{prop}[thm]{Proposition}
\newtheorem{cor}[thm]{Corollary}
\newtheorem{lem}[thm]{Lemma}

\theoremstyle{definition}
\newtheorem{defn}[thm]{Definition}

\newtheorem*{que}{Question}

\theoremstyle{remark}
\newtheorem{rem}[thm]{Remark}

\theoremstyle{plain}

\newcommand{\R}{\mathbb{R}}

\newcommand{\dimn}{\mathrm{dim}}

\newcommand{\identity}{\mathrm{id}}
\newcommand{\scal}{\mathrm{scal}}
\newcommand{\ric}{\mathrm{Ric}}
\newcommand{\trace}{\mathrm{tr}}
\newcommand{\kernel}{\mathrm{ker}}

\newcommand{\volume}{\mathrm{vol}}

\newcommand{\dv}{\text{ }dV}

\newcommand{\Diff}{\mathrm{Diff}}

\newcommand{\brk}{\text{ }}
\newcommand{\dbrk}{\text{ }\text{ }}

\newcommand{\Biggmid}{\hspace{1mm}\Bigg\vert\hspace{1mm}}

\renewcommand{\title}[1]{{\bfseries #1}\par}
\renewcommand{\author}[1]{\medskip{#1}\par\smallskip}
\newcommand{\affiliation}[1]{{\itshape #1}\par}
\newcommand{\email}[1]{E-mail:~\texttt{#1}\par}

\numberwithin{equation}{section}

\begin{document}
\begin{center}
\title{\LARGE On the stability of Einstein manifolds}
\vspace{3mm}
\author{\Large Klaus Kröncke}
\vspace{3mm}
\affiliation{Universität Potsdam, Institut für Mathematik\\Am Neuen Palais 10\\14469 Potsdam, Germany} 
\email{klaus.kroencke@uni-potsdam.de} 
\end{center}
\vspace{2mm}
\begin{abstract}Certain curvature conditions for stability of Einstein manifolds with respect to the Einstein-Hilbert action are given.
 These conditions are given in terms of quantities involving the Weyl tensor and the Bochner tensor.
In dimension six, a stability criterion involving the Euler characteristic is given.
% This notion of stability is closely related to the concept of physical stability, which comes from higher dimensional gravity theories. We also give curvature conditions for physical stability.
\end{abstract}
 \section{Introduction}
Let $M^n$ be a compact manifold of dimension $n\geq3$ and let $\mathcal{M}$ be the set of smooth Riemannian metrics on it. For any $c>0$, let $\mathcal{M}_c\subset\mathcal{M}$ be the subset of smooth Riemannian metrics of volume $c$. 
Ricci-flat metrics can be variationally characterized as critical points of the Einstein-Hilbert action
\begin{align*}S\colon\mathcal{M}&\to\R,\\
              g&\mapsto \int_M\scal_g\dv_g
\end{align*}
\cite{Hil15}. 
If the functional is restricted to some $\mathcal{M}_c$, the critical points are precisely the Einstein metrics of volume $c$. It is well-known that Einstein manifolds are neither local maximum nor minimum of the Einstein-Hilbert action on $\mathcal{M}_c$ \cite{Mut74}. 
In fact, both index and coindex of $S''$ are infinite on any Einstein space. However, there is a notion of stability which is as follows:
We say that an Einstein manifold is stable if $S''(h)\leq0$ for any $h\in \Gamma(S^2M)$ satisfying $\trace h=0$ and $\delta h=0$. Such tensors are called transverse traceless. We call the manifold strictly stable, if $S''(h)<0$ for all nonzero transverse tracless tensors.

Stability of compact Einstein metrics appears in mathematical general relativity. In \cite{AMo11}, Andersson and Moncrief
prove that the Lorentzian cone over a compact negative Einstein metric is an attractor of the Einstein flow under the assumption that the compact
Einstein metric is stable. 
Stability also appears in the context of the Ricci flow and its analysis close to Einstein metrics \cite{Ye93,CHI04,Ses06,Has12,CH13}. 
This is because the second variational formulas of Perelman's entropies on Einstein metrics are closely related to the second variational formula of the Einstein Hilbert action.

Many classes of Einstein spaces are known to be stable.
Most symmetric spaces of compact type (including the sphere and the complex projective space) are stable \cite{Koi80,CH13}. Spin manifolds admitting a nonzero parallel spinor are stable \cite{Wan91,DWW05}. K\"ahler-Einstein manifolds of nonpositive scalar curvature are stable 
\cite{DWW07}, which essentially follows from the work in \cite{Koi83}.

On the other hand, many unstable Einstein manifolds can be explicitly constructed \cite{PP84b,PP84a,GM02,GH02,GHP03,Boe05}. All these examples are of positive scalar curvature. No unstable Einstein manifolds of nonpositive scalar curvature are known which naturally leads to the 
following
\begin{que}[{{\cite[p.\ 65]{Dai07}}}]Are all compact Einstein manifolds with nonpositive scalar curvature stable?
 \end{que}
For the Ricci-flat case, this question was already asked by Kazdan and Warner \cite[p.\ 315]{KW75}.
 The statement is not true in the noncompact case since the Riemannian Schwarzschild metric \index{Riemannian Schwarzschild metric}is unstable (see \cite[Sec. 5]{GPY82}).

 Throughout this work, any manifold $M^n$ is compact and $n\geq3$ unless the contrary is explicitly asserted.
For the study of curvature conditions, we build up on an important theorem by Koiso, which states the following:
 \begin{thm}[{{\cite[Theorem 3.3]{Koi78}}}]\label{koiso}Let $(M,g)$ be an Einstein manifold with Einstein constant $\mu$. If the function $r$ satisfies
\begin{align*}
\sup_{p\in M}r(p)\leq \max\left\{-\mu,\frac{1}{2}\mu\right\},
 \end{align*}
 then $(M,g)$ is stable. If the strict inequality holds, then $(M,g)$ is strictly stable.
 \end{thm}
Here, $r:M\to\R$ is the largest eigenvalue of the curvature tensor acting on traceless symmetric $(0,2)$-tensors. The proof is based on the Bochner technique. One can estimate $r$ in a purely algebraic way in terms of sectional curvature bounds
and one gets the following corollaries as consequences thereof:
\begin{cor}[Bourguignon, unpublished]\label{pinching}           
 Let $(M,g)$ be an Einstein manifold such that the sectional curvature lies in the interval $(\frac{n-2}{3n},1]$. Then $(M,g)$ is strictly stable.
\end{cor}
\begin{cor}[{{\cite[Proposition 3.4]{Koi78}}}]\label{stabilitywhenK<0}
 Let $(M,g)$ be an Einstein manifold with sectional curvature $K<0$. Then $(M,g)$ is strictly stable.
\end{cor}
However, Corollary \ref{pinching} is ruled out for dimensions $n\geq8$ because any Einstein manifold satisfying this condition is isometric to a quotient of the round sphere. This follows from the proof of the differentiable sphere theorem \cite{BS09}.
% One gets stability by replacing the strict inequalities in the corollaries by weak inequalities. Moreover, we found out that the existence of infinitesimal Einstein deformations imposes strong conditions on the manifold. These are given in
% Proposition \ref{splittingtheorem1} and Proposition \ref{splittingtheorem2}.

Since constant curvature metrics are stable by the above, we find it convenient to formulate stability criterions in terms of the Weyl tensor. 
Let $w:M\to\R$ be the largest eigenvalue of the Weyl tensor as an operator acting on symmetric $(0,2)$-tensors.
From Koiso's Bochner formulas, we get
 \begin{thm}\label{supremeweyl}An Einstein manifold $(M,g)$ with constant $\mu$ is stable if
 \begin{align*}\left\|w\right\|_{L^{\infty}}\leq \max\left\{\mu\frac{n+1}{2(n-1)},-\mu\frac{n-2}{n-1}\right\}.
 \end{align*}
 If the strict inequality holds, then $(M,g)$ is strictly stable.
 \end{thm}
 \noindent
Using the Sobolev inequality, we find a different criterion involving an integral of this function:
 \begin{thm}\label{integralweyl}Let $(M,g)$ be an Einstein manifold with positive Einstein constant $\mu$. If
 \begin{align*}\left\|w\right\|_{L^{n/2}}\leq \mu\cdot\volume(M,g)^{2/n}\cdot\frac{n+1}{2(n-1)}\left(\frac{4(n-1)}{n(n-2)}+1\right)^{-1},
 \end{align*}
 then $(M,g)$ is stable. If the strict inequality holds, then $(M,g)$ is strictly stable.
 \end{thm}
Observe that for large dimensions, the two above conditions are close to each other.
Using the previous criterion and the Gauss-Bonnet formula in dimension six, we prove a stability criterion involving the Euler characteristic of the manifold:
 \begin{thm}\label{einsteinsixstability}Let $(M,g)$ be a positive Einstein six-manifold with constant $\mu$ and $\volume(M,g)=1$.
 If
 \begin{align*}\frac{1}{25}\left(144-\frac{12\cdot7^2\cdot 3^2}{5\cdot 11^2}\right)\mu^3\leq384\pi^3\chi(M)-48\int_M \trace (\hat{W}^3)\dv,
\end{align*}
 then $(M,g)$ is strictly stable. Here, $\hat{W}$ is the Weyl curvature operator acting on two-forms.
 \end{thm}
For K\"ahler-Einstein manifolds, the Bochner tensor plays a similar role as the Weyl tensor for general Einstein manifolds. We prove similar theorems as Theorem \ref{supremeweyl} and Theorem \ref{integralweyl} for K\"ahler-Einstein manifolds which involve
the Bochner tensor instead of the Weyl tensor. In this context, we correct a small error in \cite{IN05}.
\vspace{3mm}
\textbf{Acknowledgement.} This article is based on a part of the authors' PhD-thesis. The author would like to thank his advisor Christian B\"ar for helpful discussions. Moreover, the 
author thanks the Max-Planck Institute for Gravitational Physics for financial support.
\section{Preliminaries}
Let us first fix some notation and conventions. We define the Laplace-Beltrami operator acting on functions by $\Delta=-\trace \nabla^2$. For the Riemann curvature tensor, we use the sign convention such that
 $R_{X,Y}Z=\nabla^2_{X,Y}Z-\nabla^2_{Y,X}Z$. Given a fixed metric, we equip the bundle of $(r,s)$-tensor fields (and any subbundle) with the natural scalar product induced by the metric.
By $S^pM$, we denote the bundle of symmetric $(0,p)$-tensors.
 Let $\left\{e_1,\ldots,e_n\right\}$ be a local orthonormal frame. The divergence is the map $\delta:\Gamma(S^pM)\to\Gamma(S^{p-1}M)$, defined by
\begin{align*}\delta T(X_1,\ldots,X_{p-1})=-\sum_{i=1}^n\nabla_{e_i}T(e_i,X_1,\ldots,X_{p-1})
\end{align*}
and its adjoint $\delta^*\colon\Gamma(S^{p-1}M)\to \Gamma(S^pM)$ with respect to the natural $L^2$-scalar product is given by
\begin{align*}\delta^*T(X_1,\ldots,X_p)=\frac{1}{p}\sum_{i=0}^{p-1}\nabla_{X_{1+i}}T(X_{2+i},\ldots,X_{p+i}),
\end{align*}
where the sums $1+i,\ldots,p+i$ are taken modulo $p$.

 The second variation of $S$ at Einstein metrics was considered in \cite{Koi79}. For details, see also \cite[Chapter 4]{Bes08}.
 A useful fact for studying $S''$ is that any compact Einstein metric except the standard sphere admits the decomposition
 \begin{align}\label{decomp}T_g\mathcal{M}=\Gamma(S^2M)=C^{\infty}(M)\cdot g\oplus \delta_g^{*}(\Omega^1(M))\oplus \trace_g^{-1}(0)\cap\delta_g^{-1}(0)
 \end{align}
 and these factors are all infinite-dimensional. It turned out that this splitting is orthogonal with respect to $S''$. Thus, the second variation can be studied separately on each of these factors.

The first factor of \eqref{decomp} is the tangent space of the conformal class of $g$. It is known that $S''$ is positive on volume-preserving conformal deformations. This follows from
 \begin{equation}\label{conformalsecondvariation}S''(f\cdot g)=\frac{n-2}{2}\int_{M}\langle f,(n-1)\Delta_g f-n\mu f\rangle\dv_g
 \end{equation}
 and the following
 \begin{thm}[{{\cite[Theorem 1 and Theorem 2]{Ob62}}}]\label{obatathm}Let $(M,g)$ a compact Riemannian manifold and let $\lambda$ be the smallest nonzero eigenvalue of the Laplace
 operator acting on $C^{\infty}(M)$. Assume there exists $\mu>0$ such that $\ric(X,X)\geq \mu |X|^2$ for any vector field $X$. Then $\lambda$ satisfies the estimate
 \begin{align*}\lambda\geq \frac{n}{n-1}\mu,
 \end{align*}
 and equality holds if and only if $(M,g)$ is isometric to the standard sphere.
 \end{thm}
The second factor is the tangent space
of the orbit of the diffeomorphism group acting on $g$. By diffeomorphism invariance, $S''$ vanishes on this factor. The third factor is the space of non-trivial constant scalar curvature deformations of $g$. The tensors in the third factor are also often called transverse traceless or $TT$.
From now on, we abbreviate $TT_g=\trace_g^{-1}(0)\cap\delta_g^{-1}(0)$.
The second variation of $S$ on $TT$-tensors is given by
\begin{align*}S''(h)=-\frac{1}{2}\int_M\langle h, \nabla^*\nabla h -2\mathring{R}h\rangle\dv.
\end{align*}
Here, $\mathring{R}$ is the action of the curvature tensor on symmetric $(0,2)$-tensors, given by 
\begin{align*}\mathring{R}h(X,Y)=\sum_{i=1}^nh(R_{e_i,X}Y,e_i).
 \end{align*}
\begin{defn}We call the operator $\Delta_E:\Gamma(S^2M)\to\Gamma(S^2M)$, $\Delta_Eh=\nabla^*\nabla h -2\mathring{R}h$ the Einstein operator.
 \end{defn}
This is a self-adjoint elliptic operator and by compactness of $M$, it has a discrete spectrum. The Einstein operator preserves all components of the splitting \eqref{decomp}.
\begin{defn}\label{stability}We call a compact Einstein manifold $(M,g)$ stable, if the Einstein operator is nonnegative on $TT$-tensors and strictly stable, if it is positive on $TT$-tensors.
 We call $(M,g)$ unstable, if the Einstein operator admits negative eigenvalues on $TT$. Furthermore, elements in $\kernel(\Delta_E|_{TT})$ are called infinitesimal Einstein deformations.
\end{defn}
\begin{rem}Let $(M,g)$ be a compact Einstein manifold and let $\mathcal{C}_g$ be the set of constant scalar curvature metrics admitting the same volume as $g$. Close to $g$, this is a manifold and
\begin{align*}
 T_{g}\mathcal{C}_g=\delta_g^{*}(\Omega^1(M))\oplus TT_g,
\end{align*}
so stability precisely means that $S''$ is nonpositive on $T_g\mathcal{C}_g$.
\end{rem}
\begin{rem}
If $g_t$ is a nontrivial curve of Einstein metrics through $g=g_0$ orthogonal to $\R\cdot(g\cdot \Diff(M))$, then $\dot{g}_0$ is an infinitesimal Einstein deformation.
Evidently, an Einstein manifold is isolated in the space of Einstein structures (the set of Einstein metrics modulo diffeomorphism and rescaling) if $\Delta_E|_{TT}$ has trivial kernel.
\end{rem}
\begin{defn}An infinitesimal Einstein deformation $h$ is called integrable if there exists a curve of Einstein metrics tangent to $h$.
 \end{defn}
\section{Stability and Weyl curvature}\label{stabilityweyl}
Koiso's stability criterion (Theorem \ref{koiso}) is a first attempt to relate stability of Einstein manifolds to curvature assumptions.
Because we also work with the methods later on, we will sketch the proof of the theorem.
Let $S^2_gM$\index{$S^2_gM$} be the vector bundle
 of symmetric $(0,2)$-tensors whose trace with respect to $g$ vanishes.\index{traceless tensor}
 We define a function $r:M\to \R$ by\index{$r(p)$}
\begin{align}\label{r(p)}r(p)=\sup\left\{\frac{\langle \mathring{R}\eta,\eta\rangle_p}{|\eta|^2_p}\Biggmid\eta\in (S^2_gM)_p\right\}.
\end{align}
We now use the Bochner technique. Let the two differential operators $D_1$ and $D_2$ be defined\index{differential operator} by\index{$D_1$, a differential operator}\index{$D_2$, a differential operator}
\begin{align*}D_1h(X,Y,Z)=&\frac{1}{\sqrt{3}}(\nabla_X h(Y,Z)+\nabla_Y h(Z,X)+\nabla_Z h(X,Y)),\\
D_2h(X,Y,Z)=&\frac{1}{\sqrt{2}}(\nabla_X h(Y,Z)-\nabla_Y h(Z,X)).
\end{align*}
For the Einstein operator, we have the Bochner formulas\index{Bochner formula}
\begin{align}\label{bochner3}(\Delta_Eh,h)_{L^2}&=\left\|D_1h\right\|_{L^2}^2+2\mu \left\|h\right\|^2_{L^2}-4(\mathring{R}h,h)_{L^2}-2\left\|\delta h\right\|_{L^2}^2,\\
             \label{bochner4} (\Delta_Eh,h)_{L^2}&=\left\|D_2h\right\|_{L^2}^2-\mu \left\|h\right\|^2_{L^2}-(\mathring{R}h,h)_{L^2}+\left\|\delta h\right\|_{L^2}^2,
\end{align}
see \cite[p.~428]{Koi78} or \cite[p.~355]{Bes08} for more details.
Because of the bounds on $r$ and $\delta h=0$, we obtain either $(\Delta_Eh,h)_{L^2}\geq0$ or $(\Delta_Eh,h)_{L^2}>0$ for $TT$-tensors\index{TT@$TT$-tensor} by (\ref{bochner3}) or (\ref{bochner4}).
This proves Theorem \ref{koiso}.

 We have seen that constant curvature metrics and sufficiently pinched Einstein manifolds are stable.\index{sectional curvature!pinched}
 This motivates to prove stability theorems in terms of the Weyl tensor\index{Weyl curvature tensor} which measures the deviation of an Einstein manifold of being of constant curvature.\index{constant curvature}
Recall that on Einstein manifolds, the curvature tensor decomposes as 
 \begin{align}\label{riccidecomposition2}R=W+\frac{\mu}{2(n-1)}(g\owedge g),
 \end{align}
where $\mu$ is the Einstein constant of $g$. The tensor $W$ is the Weyl curvature tensor and $\owedge$ denotes the Kulkarni-Nomizu product of symmetric $(0,2)$-tensors, given by
 \begin{align*}(h\owedge k)(X,Y,Z,W)=h(X,W)&k(Y,Z)+h(Y,Z)k(X,W)
                                           -h(X,Z)k(Y,W)-h(Y,W)k(X,Z).
 \end{align*}
The Weyl tensor acts naturally on symmetric $(0,2)$-tensors by
 \begin{align*}\mathring{W}h(X,Y)=\sum_{i,j=1}^nW(e_i,X,Y,e_j)h(e_j,e_i),\end{align*}
and a straightforward calculation shows that the action of the Riemann tensor decomposes as
 \begin{align*}\mathring{R}h(X,Y)=\mathring{W}h(X,Y)+\frac{\mu}{n-1}\left\{g(X,Y)\text{tr}h-h(X,Y)\right\}.
   \end{align*}
 \begin{lem}\label{traceW}Let $(M,g)$ be any Riemannian manifold and let $p\in M$. The operator $\mathring{W}:(S^2M)_p\to (S^2M)_p$ is trace-free. It is indefinite\index{indefinite} as long as $W_p\neq0$.
 \end{lem}
 \begin{proof}First we compute the trace of $\mathring{W}$ acting on all symmetric $(0,2)$-tensors.
 Let $\{e_1,\ldots, e_n\}$ be an orthonormal basis of $T_pM$. Then an orthonormal basis of $(S^2M)_p$ is given by
 \begin{align*}\eta_{(ij)}=\frac{1}{\sqrt{2}}e_i^*\odot e_j^*,\qquad 1\leq i\leq j\leq n,\end{align*}
 where $\odot$ denotes the symmetric tensor product\index{symmetric tensor product}.
 Simple calculations yield
 \begin{align*}\langle \mathring{W}\eta_{(ij)},\eta_{(ij)}\rangle=-W_{ijji}.\end{align*}
 Thus,
 \begin{align*}\trace{\mathring{W}}=\sum_{1\leq i\leq j\leq n}\langle \mathring{W}\eta_{(ij)},\eta_{(ij)}\rangle=-\sum_{1\leq i\leq j\leq n}W_{ijji}=-\frac{1}{2}\sum_{i,j=1}^nW_{ijji}=0
 \end{align*}
 because the Weyl tensor has vanishing trace.
Suppose now that the operator $\mathring{W}$ vanishes, then all $W_{ijji}$ vanish. By the symmetries of the Weyl tensor\index{Weyl curvature tensor}, this already implies that $W_p$ vanishes.
 \end{proof}
\noindent
 To study the behavior of this operator, we define a function $w:M\to\R$ by \index{$w(p)$}
 \begin{align}\label{functionw}w(p)=\sup\left\{\frac{\langle \mathring{W}\eta,\eta\rangle_p}{|\eta|_p^2}\Biggmid\eta\in (S^2M)_p\right\}.
 \end{align}
 Thus, $w(p)$ is the largest eigenvalue of the action $\mathring{W}\colon(S^2M)_p\to (S^2M)_p$. Lemma \ref{traceW} implies that the function $w$ is nonnegative.
 Since $\mathring{W}g=0$, $w(p)$ is also the largest eigenvalue of $\mathring{W}$ restricted to $(S_g^2M)_p$.

 The decomposition of $\mathring{R}$ allows us to estimate the smallest eigenvalue of $\Delta_E$ acting on $TT$-tensors \index{TT@$TT$-tensor}in terms of the function $w$.
From (\ref{bochner3}), we obtain 
 \begin{align*}(\Delta_Eh,h)&=\left\|D_1h\right\|^2_{L^2}+2\mu\left\|h\right\|_{L^2}^2-4(\mathring{R}h,h)\\
                               &\geq 2\mu\left\|h\right\|^2_{L^2}+4\frac{\mu}{n-1}\left\|h\right\|^2_{L^2}-4(\mathring{W}h,h)\\
                               &\geq\left[2\mu\frac{n+1}{n-1}-4\left\|w\right\|_{\infty}\right]\left\|h\right\|^2_{L^2},
 \end{align*}
 and similarly from (\ref{bochner4}), 
 \begin{align*}(\Delta_E h,h)&=\left\|D_2h\right\|^2_{L^2}-\mu\left\|h\right\|_{L^2}^2-(\mathring{R}h,h)\\
                               &\geq -\mu\left\|h\right\|^2_{L^2}+\frac{\mu}{n-1}\left\|h\right\|^2_{L^2}-(\mathring{W}h,h)\\
                               &\geq\left[-\mu\frac{n-2}{n-1}-\left\|w\right\|_{\infty}\right]\left\|h\right\|^2_{L^2}.
 \end{align*}
 \begin{prop}\label{Weylproposition}Let $(M,g)$ be Einstein with constant $\mu$ and let $\lambda$ be the smallest eigenvalue of $\Delta_{E}|_{TT}$. Then
 \begin{align*}\lambda\geq\max\left\{2\mu\frac{n+1}{n-1}-4\left\|w\right\|_{\infty},-\mu\frac{n-2}{n-1}-\left\|w\right\|_{\infty}\right\}.
 \end{align*}
 \end{prop}
\noindent
 Now, Theorem \ref{supremeweyl} is an obvious consequence.
The following result is similar in its form:
\begin{thm}[{{\cite[Theorem 1]{IN05}}}]Let $(M,g)$ be a compact, connected oriented\index{oriented} Einstein manifold with negative Einstein constant $\mu$. If
\begin{align}\label{INcondition}\sup_{p\in M}\overline{w}(p)<-\frac{\mu}{n-1},
\end{align}
then $(M,g)$ is strictly stable. Here, $\overline{w}(p)$ be the largest eigenvalue of the Weyl curvature operator at $p\in M$.
\end{thm}
\noindent
However, this condition is equivalent to the condition that the Riemann curvature operator on $(M,g)$ is negative. By Corollary \ref{stabilitywhenK<0}, strict stability holds under the weaker condition of negative sectional curvature.

It seems not convenient to formulate stability criterions in terms of the curvature operator. In the proof of the above theorem, the very rough estimate $\max W_{ijji}\leq \overline{w}(p)$ is used and we find no other way to estimate
$w(p)$ in terms of $\overline{w}(p)$.

   We now give a different stability criterion involving an integral of the function $w$. The main tool we use here is the Sobolev inequality which holds for Yamabe metrics.
 Recall that a metric is called Yamabe if it realizes the Yamabe metric in its conformal class, given by
\begin{align*}Y([g])=\inf_{\tilde{g}\in[g]}\volume(M,\tilde{g})^{(2-n)/n}\int_M \scal_{\tilde{g}}\dv_{\tilde{g}}.
\end{align*}
 \begin{prop}[Sobolev inequality]\index{Sobolev!inequality}Let $(M,g)$ be a Yamabe metric\index{Yamabe!metric} in a conformal class\index{conformal!class} and suppose that $\volume(M,g)=1$. Then for any $f\in H^1(M)$,\index{$H^1(M)$, Sobolev space of functions on $M$}
 \begin{align}\label{sobolev}4\frac{n-1}{n-2}\left\|\nabla f\right\|^2_{L^2}\geq\scal\left\{\left\|f\right\|^2_{L^p}-\left\|f\right\|^2_{L^2}\right\},
 \end{align}
 where $p=2n/(n-2)$.
 \end{prop}
 \begin{proof}This follows easily from the definition of Yamabe metrics, see e.g.\ \cite[p.\ 140]{ItSa02}.
 \end{proof}
 \begin{rem}The inequality holds if $f$ is replaced by any tensor $T$ because of Kato's inequality\index{Kato's inequality}
 \begin{align}|\nabla|T||\leq|\nabla T|.
 \end{align}
 \end{rem}
\noindent
 Since any Einstein metric is Yamabe (see e.g.\ \cite[p.~329]{LeB99}), the Sobolev inequaliy\index{Sobolev!inequality} holds in this case.
 Now we are ready to prove Theorem \ref{integralweyl}:
 \begin{proof}[Proof of Theorem \ref{integralweyl}]Both sides of the inequality in the statement are scale-invariant\index{scale-invariance}, see Lemma \ref{scalingofb} below.
 Therefore, we may assume $\volume(M,g)=1$ from now on.
 First, we estimate the largest eigenvalue of the Weyl tensor action by
 \begin{align*}(\mathring{W}h,h)_{L^2}\leq\int_M |w| |h|^2 \dv                                      &\leq\left\|w\right\|_{L^{n/2}}\left\|h\right\|^2_{L^{2n/(n-2)}}
                                      \leq\left\|w\right\|_{L^{n/2}}\left(4\frac{n-1}{\mu n(n-2)}\left\|\nabla h\right\|^2_{L^2}+\left\|h\right\|^2_{L^2}\right).
 \end{align*}
 We used the H\"older inequality\index{H\"older inequality} and the Sobolev inequality\index{Sobolev!inequality}. With the estimate obtained, we can proceed as follows:
 \begin{align*}(\Delta_Eh,h)_{L^2}&=\left\|\nabla h\right\|^2_{L^2}-2(\mathring{R}h,h)_{L^2}\\
                                           &=\left\|\nabla h\right\|^2_{L^2}+2\frac{\mu}{n-1}\left\|h\right\|^2_{L^2}-2(\mathring{W}h,h)_{L^2}\\
                                           &\geq\left\|\nabla h\right\|^2_{L^2}+2\frac{\mu}{n-1}\left\|h\right\|^2_{L^2}
                                            -2 \left\|w\right\|_{L^{n/2}}\left(4\frac{n-1}{\mu n(n-2)}\left\|\nabla h\right\|^2_{L^2}+\left\|h\right\|^2_{L^2}\right)\\
                                           &=\left(1-8\frac{n-1}{\mu n(n-2)}\left\|w\right\|_{L^{n/2}}\right)\left\|\nabla h\right\|^2_{L^2}
                                            +2\left(\frac{\mu}{n-1}-\left\|w\right\|_{L^{n/2}}\right)\left\|h\right\|^2_{L^2}
 \end{align*}
 The first term on the right hand side is nonnegative by the assumption on $w$.
 It remains to estimate $\left\|\nabla h\right\|^2_{L^2}$. This can be done by using (\ref{bochner3}). We have
 \begin{align*}\left\|\nabla h\right\|^2_{L^2}&=\left\|D_1 h\right\|^2_{L^2}+2\mu \left\| h\right\|_{L^2}^2-2(\mathring{R}h,h)_{L^2}\\
                                              &\geq2\mu \left\| h\right\|_{L^2}^2+2\frac{\mu}{n-1}\left\| h\right\|_{L^2}^2-2(\mathring{W}h,h)_{L^2}\\
                                              &\geq2\mu\frac{n}{n-1}\left\| h\right\|_{L^2}^2-2\left\|w\right\|_{L^{n/2}}
                                              \left(4\frac{n-1}{\mu n(n-2)}\left\|\nabla h\right\|^2_{L^2}+\left\|h\right\|^2_{L^2}\right)\\
                                              &=2\left(\mu\frac{n}{n-1}-\left\|w\right\|_{L^{n/2}}\right)\left\| h\right\|^2_{L^2}-8
                                              \frac{n-1}{\mu n(n-2)}\left\|w\right\|_{L^{n/2}}\left\|\nabla h\right\|^2_{L^2},
 \end{align*} 
 and therefore, $\left\|\nabla h\right\|^2_{L^2}$ can be estimated by
 \begin{align*}\left\|\nabla h\right\|^2_{L^2}\geq 2\left(\mu\frac{n}{n-1}-\left\|w\right\|_{L^{n/2}}\right)\left(1+8\frac{n-1}{\mu n(n-2)}\left\|w\right\|_{L^{n/2}}\right)^{-1}
 \left\|h\right\|^2_{L^2}.
 \end{align*}
 Combining these arguments, we obtain 
 \begin{equation}\begin{split}\label{integralestimate}(\Delta_Eh,h)_{L^2}\geq&\bigg\{2\left(1-8\frac{n-1}{\mu n(n-2)}\left\|w\right\|_{L^{n/2}}\right)\left(\mu\frac{n}{n-1}-\left\|w\right\|_{L^{n/2}}\right)\cdot\\
                                      &\left(1+8\frac{n-1}{\mu n(n-2)}\left\|w\right\|_{L^{n/2}}\right)^{-1}+2\left(\frac{\mu}{n-1}-\left\|w\right\|_{L^{n/2}}\right)\bigg\}\left\|h\right\|^2_{L^2}.
\end{split} 
\end{equation}
 The manifold $(M,g)$ is stable if the right-hand side of this inequality is nonnegative. It is elementary to check that this is equivalent to
 \begin{align*}\left\|w\right\|_{L^{n/2}}\leq \mu\frac{n+1}{2(n-1)}\left(\frac{4(n-1)}{n(n-2)}+1\right)^{-1}.
  \end{align*}
The assertion about strict stability is also immediate.
 \end{proof}
 \begin{lem}\label{scalingofb}The $L^{n/2}$-Norm of the function $w$ is conformally invariant\index{conformally!invariant}.
 \end{lem}
 \begin{proof}Let $g,\tilde{g}$ be conformally equivalent\index{conformally!equivalent}, i.e.\ $\tilde{g}=f\cdot g$ for a smooth positive function $f$. 
Let $W$ and $\tilde{W}$ be the Weyl tensors\index{Weyl curvature tensor} of the metrics $g$ and $\tilde{g}$, respectively.
 It is well-known that $\tilde{W}=f\cdot W$ when considered as $(0,4)$-tensors. Therefore,
 \begin{align*}\langle\mathring{\tilde{W}}h,h\rangle_{\tilde{g}}
=f^{-3}\langle \mathring{W}h,h\rangle_g.
 \end{align*}
 Furthermore, we have 
 \begin{align*}|h|_{\tilde{g}}^2=f^{-2}|h|_{g}^2,\qquad \dv_{\tilde{g}}=f^{n/2}\dv_g.
 \end{align*}
 We now see that $\tilde{w}=f^{-1}w$ and
 \begin{align*}\left\|\tilde{w}\right\|^{2/n}_{L^{n/2}(\tilde{g})}=\int_M \tilde{w}^{n/2} \dv_{\tilde{g}}=\int_M w^{n/2} \dv_{g}=
 \left\|w\right\|^{2/n}_{L^{n/2}(g)},
 \end{align*}
 which shows the lemma.
 \end{proof}
 \begin{cor}Let $(M,g)$ be a Riemannian manifold and let $Y([g])$ be the Yamabe constant\index{Yamabe!constant}\index{$Y([g])$, Yamabe constant of $[g]$}
 of the conformal class\index{conformal!class} of $g$. If 
 \begin{equation}\left\|w\right\|_{L^{n/2}(g)}\leq Y([g])\frac{n+1}{2n(n-1)}\cdot\left(\frac{4(n-1)}{n(n-2)}+1\right)^{-1},
 \end{equation}
 any Einstein metric in the conformal class of $g$ is stable.
 \end{cor}
 \begin{proof}Suppose that $\tilde{g}\in[g]$ is Einstein. 
 We know that $\tilde{g}$ is a \index{Yamabe!metric}Yamabe metric in the conformal class of $g$. By the definition of the Yamabe constant, the Einstein constant of $\tilde{g}$ equals
 \begin{align*}\mu=\frac{1}{n}\cdot Y([g])\cdot \volume(M,g)^{2/n},
 \end{align*}
 and Lemma \ref{scalingofb} yields 
 \begin{align*}\left\|\tilde{w}\right\|_{L^{n/2}(\tilde{g})}=\left\|w\right\|_{L^{n/2}(g)}\leq \mu\cdot\volume(M,g)^{2/n}\cdot\frac{n+1}{2(n-1)}\cdot\left(\frac{4(n-1)}{n(n-2)}+1\right)^{-1}.
 \end{align*}
 The assertion now follows from Theorem \ref{integralweyl}.
 \end{proof}
 By Theorem \ref{integralweyl} and the Cauchy-Schwarz inequality, any positive Einstein manifold of unit volume\index{volume!unit} is stable, if
  \begin{align}\label{integralweyl2}\left\|W\right\|_{L^{n/2}}\leq \mu\cdot\frac{n+1}{2(n-1)}\left(\frac{4(n-1)}{n(n-2)}+1\right)^{-1}.
 \end{align}
\noindent
On the other hand, we have the following isolation theorem for the Weyl tensor:
 \begin{thm}[{{\cite[Main theorem]{ItSa02}}}]Let $(M,g)$ be a compact connected, oriented Einstein-manifold, $n\geq4$, with positive
 Einstein constant $\mu$ and of unit-volume. Then there exists a constant $C(n)$, depending only on $n$,
 such that if the inequality $\left\|W\right\|_{L^{n/2}}<C(n)\mu$ holds, then $W=0$ so that $(M,g)$ is a finite isometric quotient of the sphere.
 \end{thm}
\noindent
 A careful investigation of the proof shows that $W$ vanishes if
 \begin{align}\label{weylisolation}\left\|W\right\|_{L^{n/2}}\leq\left\{\begin{array}{l l}\frac{n(n-2)}{24(n-1)}\mu & \quad\text{if }4\leq n\leq9,\\
                                              \frac{1}{3}\mu & \quad\text{if }n\geq10.
 \end{array} \right.
 \end{align}
\noindent
A comparison of the last two inequalities shows that \eqref{integralweyl2} is not ruled out by the above isolation theorem.
In dimension $4$, we have another isolation theorem, proven with different techniques:
 \begin{thm}[{{\cite[Theorem 1]{GLeB99}}}]\label{GLtheorem}Let $(M,g)$ be a compact oriented Einstein 4-manifold with constant $\mu>0$ and let $W^+$ be the self-dual part of the Weyl tensor. If $W^+\not\equiv0$,
 then 
\begin{align*}\int_M |W^+|^2  \dv\geq \int_M\frac{8\mu^2}{3}\dv,
\end{align*}
 with equality if and only if $\nabla W^+\equiv0$.
 \end{thm}
\noindent
Obviously, the same gap theorem holds for the whole Weyl tensor.
By passing to the orientation covering, we see that the same gap also holds for the Weyl tensor on non-orientable manifolds. Unfortunately, this theorem rules \eqref{integralweyl2} out.
 \section{Six-dimensional Einstein manifolds}\label{stabilitysix}
In this section, we compute an explicit representation of the Gauss-Bonnet formula \index{Gauss-Bonnet formula}for six-dimensional Einstein manifolds. We use this representation to show a stability criterion for Einstein manifolds involving the 
Euler characteristic\index{Euler characteristic}.

 The generalized Gauss-Bonnet formula for a compact Riemannian manifold $(M,g)$ of dimension $n=2m$ is
 \begin{align*}\chi(M)=&\frac{(-1)^m}{2^{3m}\pi^m m!}\int_{M}\Psi_g\dv.
 \end{align*}
The function $\Psi_g$ is defined as
\begin{align*}\Psi_g=\sum_{\sigma,\tau\in S_m}\mathrm{sgn}(\sigma)\mathrm{sgn}(\tau)R_{\sigma(1)\sigma(2)\tau(1)\tau(2)}\ldots R_{\sigma(n-1)\sigma(n)\tau(n-1)\tau(n)},
\end{align*}
 where the coefficients are taken with respect to an orthonormal basis (see e.g.\ \cite[Theorem 4.1]{Zhu00}).\index{$S_m$, symmetric group}
 In dimension four, this yields the nice formula
 \begin{equation}\label{gaussbonnetdimensionfour}\chi(M)=\frac{1}{32\pi^2}\int_M (|W|^2+|Sc|^2-|U|^2)\dv
 \end{equation}
 (see also \cite[p.\ 161]{Bes08}).
 Here, $Sc=\frac{\scal}{2n(n-1)}g\owedge g$\index{$Sc$, scalar part of $R$} is the scalar part and $U=\frac{1}{n-2}\ric^0\owedge g$\index{$U$, traceless Ricci part of $R$} is the traceless Ricci part of the curvature tensor. 
 Due to different conventions for the norm of curvature tensors, formula \eqref{gaussbonnetdimensionfour} often appers with the factor $\frac{1}{8\pi^2}$ instead of $\frac{1}{32\pi^2}$.
 On Einstein manifolds, 
 we have $U=0$ and the Gauss-Bonnet formula simplifies to
 \begin{equation}\label{dim4euler}\chi(M)=\frac{1}{32\pi^2}\int_M \left(|W|^2+\frac{8}{3}\mu^2\right)\dv
 \end{equation}
where $\mu$ is the Einstein constant. As a nice consequence, we obtain a topological condition for the existence of Einstein metrics:
 \begin{thm}[{{\cite[p.\ 41]{Ber65}}}]Every compact 4-manifold carrying an Einstein metric $g$ satisfies the inequality
 \begin{align*}\chi(M)\geq0.
 \end{align*}
 Moreover, $\chi(M)=0$ if and only if $(M,g)$ is flat.
 \end{thm}
Another consequence of (\ref{dim4euler}) is the following: Let $(M,g)$ be of unit volume\index{volume!unit}. Then there exists a constant $C>0$ such that, if $\mu\geq C\cdot \sqrt{\chi(M)}$, the Weyl curvature satisfies $\left\|W\right\|_{L^2}\leq \frac{1}{3}\mu$. This implies
stability by Theorem \ref{integralweyl}. Unfortunately, the same condition on the Weyl tensor\index{Weyl curvature tensor} already implies that it vanishes, as we discussed in the last section.

 In dimension six, an explicit representation of the Gauss-Bonnet formula\index{Gauss-Bonnet formula} is given by
 \begin{align*}\chi(M)=&\frac{1}{384\pi^3}\int_{M}\{\scal^3-12\scal|\ric|^2+3\scal|R|^2+16\langle\ric,\ric\circ\ric\rangle\\
                      &-24\ric^{ij}\ric^{kl}R_{ikjl}-24\ric_i^{\text{ }j}R^{iklm}R_{jklm}+8R^{ijkl}R_{imkn}R^{\text{ }n\text{ }m}_{j\text{ }l\text{ }}-2R^{ijkl}R^{\text{ }\text{ }mn}_{ij}R_{klmn}\}\dv
 \end{align*}
 (see \cite[Lemma 5.5]{Sak71}).
 When $(M,g)$ is Einstein, this integral is equal to
 \begin{equation}\begin{split}\label{gaussbonnetdimensionsix}\chi(M)=&\frac{1}{384\pi^3}\int_{M}\{24\mu^3-6\mu|R|^2+8R^{ijkl}R_{imkn}R^{\text{ }n\text{ }m}_{j\text{ }l\text{ }}-2R^{ijkl}R^{\dbrk mn}_{ij}R_{klmn}\}\dv.
 \end{split}\end{equation}
 \begin{lem}\label{sakailemma}If $(M,g)$ is a compact Einstein manifold with constant $\mu$,
 \begin{align*}\left\|\nabla R\right\|_{L^2}^2=-\int_{M}\{4R^{ijkl}R^{\brk m\brk n}_{i\brk k}R_{jnlm}+2R^{ijkl}R^{\dbrk mn}_{ij}R_{klmn}+2\mu|R|^2\}\dv.
 \end{align*}
 \end{lem}
  \begin{proof}
This is \cite[(2.15)]{Sak71} in the special case of Einstein metrics.
 \end{proof}
\noindent
Note that we translated the formulas from \cite{Sak71} to our sign convention for the curvature tensor.
\begin{prop}\label{einsteinsix}Let $(M,g)$ be an Einstein six-manifold with constant $\mu$. Then
\begin{align*}\chi(M)&=\frac{1}{384\pi^3}\int_M\left\{-\frac{14}{5}\mu|W|^2-2|\nabla W|^2+\frac{144}{25}\mu^3+48\trace (\hat{W}^3)\right\}\dv.
\end{align*}
 Here, $\hat{W}^3=\hat{W}\circ \hat{W}\circ \hat{W}$, where $\hat{W}$ is the Weyl curvature operator\index{Weyl curvature operator} acting on $2$-forms\index{2@$2$-form}.
\end{prop}
\begin{proof}By Lemma \ref{sakailemma}, (\ref{gaussbonnetdimensionsix}) can be rewritten as
 \begin{align*}384\pi^3\chi(M)=\int_M\{24\mu^3-10\mu|R|^2-2|\nabla R|^2-6 R^{ijkl}R^{\dbrk mn}_{ij}R_{klmn}\}\dv.
\end{align*}
 Moreover, $\nabla W=\nabla R$ because the difference $R-W=Sc$ is a parallel tensor. Thus,
 \begin{align*}384\pi^3\chi(M)&=\int_M\{24\mu^3-10\mu|R|^2-2|\nabla W|^2 -6 R^{ijkl}R^{\dbrk mn}_{ij}R_{klmn}\}\dv\\
                              &=\int_M\left\{24\mu^3-10\mu\left(\frac{12\mu^2}{5}+|W|^2\right)-2|\nabla W|^2 -6 R^{ijkl}R^{\dbrk mn}_{ij}R_{klmn}\right\}\dv\\
                              &=\int_M\{-10\mu|W|^2-2|\nabla W|^2 -6 R^{ijkl}R^{\dbrk mn}_{ij}R_{klmn}\}\dv.
 \end{align*}
 Now we analyse the last term on the right hand side. Recall that the Riemann curvature operator \index{Riemann curvature operator}$\hat{R}$\index{$\hat{R}$ Riemann curvature operator} and the Weyl curvature operator\index{Weyl curvature operator} $\hat{W}$ are defined by
\begin{align*}\langle\hat{R}(X\wedge Y),Z\wedge V\rangle&=R(Y,X,Z,V),\\
\langle\hat{W}(X\wedge Y),Z\wedge V\rangle&=W(Y,X,Z,V).
\end{align*}
Let $\left\{e_1,\ldots, e_n\right\}$ be a local orthonormal frame of $TM$. Then $\left\{e_i\wedge e_j\right\}$, $i<j$ is a local orthonormal frame of $\Lambda^2M$. A straightforward calculation shows
 \begin{align*}          -6 \sum_{i,j,k,l,m,n}R_{ijkl}R_{ijmn}R_{klmn}
           =&-48\sum_{\substack{i<j,k<l,m<n}}R_{ijkl}R_{ijmn}R_{klmn}=48 \trace \hat{R}^3,
 \end{align*}
where the coefficients of $R$ are taken with respect to the orthonormal frame.
 The decomposition (\ref{riccidecomposition2}) of the $4$-curvature tensor induces the decomposition $\hat{R}=\hat{W}+\frac{\mu}{5}\identity_{\Lambda^2M}$. This yields
\begin{align*}48 \trace \hat{R}^3&=48 \left\{\trace (\hat{W}^3) + 3\frac{\mu}{5}\trace (\hat{W}^2)+ 3\frac{\mu^2}{25}\trace \hat{W}+\frac{\mu^3}{125}\trace (\identity_{\Lambda^2M})\right\}\\
                           &=48\trace (\hat{W}^3)+\frac{36}{5}\mu|W|^2+\frac{144}{25}\mu^3.
\end{align*}
Inserting this in the above formula finishes the proof.
\end{proof}
 \begin{proof}[Proof of Theorem \ref{einsteinsixstability}]By the Sobolev inequality\index{Sobolev!inequality},
 \begin{align*}\left\|W\right\|_{L^3}^2\leq \frac{5}{6\mu}\left\|\nabla W\right\|^2_{L^2}+\left\|W\right\|^2_{L^2}.
 \end{align*}
Therefore we have, by Proposition \ref{einsteinsix}
 \begin{align*}               384\pi^3\chi(M)&<-\frac{12}{5}\mu\left\|W\right\|^2_{L^2}-2\left\|\nabla W\right\|_{L^2}^2+\frac{144}{25}\mu^3+48\int_M\trace (\hat{W}^3)\dv\\
                              &\leq-\frac{12}{5}\mu\left\|W\right\|_{L^3}^2+\frac{144}{25}\mu^3+48\int_M \trace (\hat{W}^3)\dv.
 \end{align*}
 Now if $\mu$ satisfies the estimate of the statement in the theorem, we obtain
 \begin{align*}\frac{12}{5}\mu\left\|W\right\|^2_{L^3}&<\frac{144}{25}\mu^3-384\pi^3\chi(M)+48\int_M \trace (\hat{W}^3)\dv\\
                                                      &\leq \frac{144}{25}\mu^3-\frac{1}{25}\left(144-\frac{12 \cdot7^2\cdot3^2}{5\cdot11^2}\right)\mu^3
                                                      =\frac{12\mu}{5}\frac{7^2\cdot3^2}{5^2\cdot11^2}\mu^2,
 \end{align*}
 which is equivalent to
 \begin{align*}\left\|W\right\|_{L^3}< \frac{7\cdot3}{5\cdot11}\mu.
 \end{align*}
 By Theorem \ref{integralweyl} and the Cauchy-Schwarz inequality, $(M,g)$ is strictly stable.
 \end{proof}
\section{K\"ahler manifolds}\label{Kahlermanifolds}
 Here, we prove stability criterions for K\"ahler-Einstein manifolds\index{K\"ahler-Einstein manifold} in terms of the Bochner curvature tensor\index{Bochner curvature tensor}, which is an analogue of the Weyl tensor\index{Weyl curvature tensor}.
 \begin{defn}Let $(M,g)$ be a Riemannian manifold of even dimension. An almost complex structure on $M$ is an endorphism $J:TM\to TM$ such that\index{$J$, almost complex structure}
 $J^2=-\identity_{TM}$. If $J$ is parallel and $g$ is hermitian, i.e.\ $g(JX,JY)=g(X,Y)$,  we call the triple $(M,g,J)$ a K\"ahler manifold\index{K\"ahler manifold}.
 If $(M,g)$ is Einstein, we call $(M,g,J)$ Einstein-K\"ahler.
 \end{defn}
The bundle of traceless symmetric $(0,2)$-tensors splits into hermitian and skew-hermitian\index{skew-hermitian} ones, i.e.\ we have $S^2_gM=H_1\oplus H_2$, 
where\index{$H_1$, space of hermitian tensors}\index{$H_2$, space of skew-hermitian tensors}
   \begin{align*}H_1&=\left\{h\in S^2_gM\mid h(X,Y)=h(JX,JY)\right\},\\
 H_2&=\left\{h\in S^2_gM\mid h(X,Y)=-h(JX,KY)\right\}.
 \end{align*}
 Stability of K\"ahler-Einstein manifolds was studied in \cite{Koi83,IN05,DWW07}. We sketch the ideas of \cite{Koi83} in the following.
 It turns out that the Einstein operator preserves the splitting $\Gamma(H_1)\oplus\Gamma(H_2)$. Therefore to show that a K\"ahler-Einstein manifold is stable it is sufficient to show that the restriction
 of $\Delta_E$ to the subspaces $\Gamma(H_1)$ and $\Gamma(H_2)$ is positive semidefinite, respectively. In fact, we can use the K\"ahler structure to conjugate the Einstein operator to other operators.
 If $h_1\in H_1$, we define a $2$-form\index{2@$2$-form} by
 \begin{align*}\phi(X,Y)=h_1\circ J(X,Y)=h_1(X,J(Y)).
 \end{align*}
 We have
 \begin{align}\label{H1formula}\Delta_H\phi=(\Delta_E h_1)\circ J+2\mu\phi,
 \end{align}
 where $\Delta_H$ is the Hodge Laplacian\index{Laplacian!Hodge}\index{$\Delta_H$, Hodge Laplacian} on $2$-forms and $\mu$ is the Einstein constant. Since $\Delta_H$ is nonnegative, $\Delta_E$ is nonnegative on $\Gamma(H_1)$, if $\mu\leq0$.
 For $h_2\in H_2$, we define a symmetric endomorphism $I:TM\to TM$ by
 \begin{align*}g\circ I=h_2\circ J,
 \end{align*}
and since $IJ+JI=0$, we may consider $I$ as a $T^{1,0}M$-valued $1$-form of type $(0,1)$.
 We have the formula\index{$\Delta_C$, complex Laplacian}
 \begin{align}\label{H2formula}g\circ(\Delta_CI)=(\Delta_Eh_2)\circ J,
 \end{align}
 where $\Delta_C$ is the complex Laplacian\index{Laplacian!complex}. Thus, the restriction of the Einstein operator to $\Gamma(H_2)$ is always nonnegative, since $\Delta_C$ is.
As a consequence, we have
 \begin{cor}[{{\cite[Corollary 1.2]{DWW07}}}]Any compact K\"ahler-Einstein manifold with nonpositive Einstein constant is stable.
 \end{cor}
Using (\ref{H1formula}) and (\ref{H2formula}), $\dimn(\kernel\Delta_{E}|_{TT})$ can be expressed in terms of certain cohomology classes\index{cohomology class} (see \cite[Corollary 9.4]{Koi83} or \cite[Proposition 12.98]{Bes08}).
Moreover, integrability of infinitesimal Einstein deformations can be related to integrability of infinitesimal complex deformations\index{infinitesimal complex deformation} (\cite[Proposition 10.1]{Koi83} and \cite[Theorem 3]{IN05}).

 We discuss conditions under which a K\"ahler-Einstein manifold is
 strictly stable in the nonpositive case and stable in the positive case. This can be described in terms of the
 Bochner curvature tensor which has similar properties as the Weyl tensor.
  \begin{defn}[Bochner curvature tensor]Let $(M,g,J)$ be a K\"ahler manifold and let $\omega(X,Y)=g(J(X),Y)$ be the K\"ahler form. The Bochner curvature tensor\index{Bochner curvature tensor} is defined by\index{$B$, Bochner curvature tensor}
 \begin{align*}B=&R+\frac{\scal}{2(n+2)(n+4)}\left\{g\owedge g+\omega\owedge \omega-4\omega\otimes \omega\right\}\\
                 &-\frac{1}{n+4}\left\{\ric\owedge g+(\ric\circ J)\owedge \omega-2(\ric\circ J)\otimes \omega-2 \omega\otimes(\ric\circ J)\right\}
 \end{align*}
 (see e.g.\ \cite[p.~229]{IK04}). 
 \end{defn}
 \noindent
 The Bochner curvature tensor posesses the same symmetries as the Riemann tensor and in addition, any of its traces vanishes.
 If $(M,g)$ is K\"ahler-Einstein\index{K\"ahler-Einstein manifold}, the Bochner tensor is
 \begin{align*}B=R-\frac{\mu}{2(n+2)}\left\{g\owedge g+\omega\owedge\omega-4\omega\otimes\omega\right\},
 \end{align*}
 where $\mu$ is the Einstein constant (see e.g.\ \cite[p.~229]{IK04} and mind the different sign convention for the curvature tensor). 
 The Bochner tensor acts naturally on symmetric $(0,2)$-tensors by\index{$\mathring{B}$, Bochner curvature action on $S^2M$}
 \begin{align*}\mathring{B}h(X,Y)=\sum_{i,j=n}^n B(e_i,X,Y,e_j) h(e_i,e_j),
 \end{align*}
 where $\left\{e_1,\ldots,e_n\right\}$ is an orthonormal basis.
 Let\index{$b^+(p)$}
  \begin{align*}b^+(p)=\left\{\frac{\langle\mathring{B}\eta,\eta\rangle}{|\eta|^2}\Biggmid\eta\in (H_1)_p\right\}.
  \end{align*}
 For K\"ahler-Einstein manifolds with negative Einstein constant, it was proven by M.\ Itoh and T.\ Nakagawa that they are strictly stable if the Bochner tensor \index{Bochner curvature tensor}
 is small.
 \begin{thm}[{{\cite[Theorem 4.1]{IN05}}}]\label{itohnagakawa}Let $(M,g,J)$ be a compact K\"ahler-Einstein manifold\index{K\"ahler-Einstein manifold} with negative Einstein constant $\mu$. If the Bochner curvature tensor
 satisfies
 \begin{align}\label{itohnagakawa_est}\left\|b^+\right\|_{L^{\infty}}<-\mu\frac{n}{n+2},
 \end{align}
 then $g$ is strictly stable.
 \end{thm}
 However, an error occured in the calculations and the result is slightly different.
 Therefore, we redo the proof. By straightforward calculation, 
 \begin{align}\label{bochnerdecomposition}\langle \mathring{R}h,h\rangle=\langle \mathring{B}h,h\rangle-\frac{\mu}{n+2}\{|h|^2-3\sum_{i,j}h(e_i,e_j)h(J(e_i),J(e_j))\}.
 \end{align}
 In particular,
 \begin{align*}\langle \mathring{R}h_1,h_1\rangle=\langle \mathring{B}h_1,h_1\rangle+2\frac{\mu}{n+2}|h_1|^2
 \end{align*}
 for $h_1\in H_1$ and
 \begin{align*}\langle \mathring{R}h_2,h_2\rangle=\langle \mathring{B}h_2,h_2\rangle-4\frac{\mu}{n+2}|h_2|^2
 \end{align*}
 for $h_2\in H_2$.  By (\ref{H1formula}), $\Delta_E$ is positive definite on $\Gamma(H_1)$ so it remains to consider $\Gamma(H_2)$.
 By (\ref{bochner4}),
 \begin{align*}(\Delta_E h_2,h_2)_{L^2}&=\left\|D_2 h_2\right\|^2_{L^2}-\mu\left\| h_2\right\|_{L^2}^2
 -(h_2,\mathring{R}h_2)_{L^2}+\left\|\delta h_2\right\|^2_{L^2}.\\
&\geq-\mu\left\| h_2\right\|_{L^2}^2 -(h_2,\mathring{B}h_2)_{L^2}+\frac{4\mu}{n+2}\left\|h_2\right\|^2_{L^2}\\
&\geq-\mu\frac{n-2}{n+2}\left\| h_2\right\|_{L^2}^2 -\left\|b^+\right\|_{L^{\infty}}\left\| h_2\right\|_{L^2}^2.
\end{align*}
 \begin{rem}Theorem \ref{itohnagakawa} is true if we replace (\ref{itohnagakawa_est}) by
 \begin{align}\label{itohnagakawa_est2}\left\|b^+\right\|_{L^{\infty}}<-\mu\frac{n-2}{n+2}.
 \end{align}
 \end{rem}
\noindent
Now, let us turn to positive K\"ahler-Einstein manifolds. We will use Bochner formula\index{Bochner formula} (\ref{bochner3}). Unfortunately, we cannot make use of the vector bundle splitting $S_g^2M=H_1\oplus H_2$.
In order to apply (\ref{bochner3}), we need the condition $\delta h=0$, which is not preserved by the splitting into hermitian\index{hermitian} and skew-hermitian\index{skew-hermitian} tensors.
 Let\index{$b(p)$}
 \begin{align}\label{functionb}b(p)&=\sup\left\{\frac{\langle \mathring{B}\eta,\eta\rangle}{|\eta|^2}\Biggmid\eta\in (S^2M)_p\right\}.
 \end{align}
 Since the trace of the Bochner tensor vanishes, $\mathring{B}\colon(S^2M)_p\to (S^2M)_p$ has also vanishing trace (this follows from the same arguments as used in the proof of Lemma \ref{traceW}).
 Thus, $b$ is nonnegative.
 \begin{thm}\label{bochnerthm1}Let $(M,g,J)$ be K\"ahler-Einstein\index{K\"ahler-Einstein manifold} with positive Einstein constant $\mu$. If
 \begin{align*}\left\|b\right\|_{L^{\infty}}\leq \frac{\mu(n-2)}{2(n+2)},
 \end{align*}
 then $(M,g)$ is stable.
 \end{thm}
 \begin{proof}Let $h\in TT$.
By (\ref{bochnerdecomposition}) and the Cauchy-Schwarz inequality\index{Cauchy-Schwarz inequality},
\begin{align*}\langle\mathring{R}h,h\rangle\leq\langle\mathring{B}h,h\rangle+2\frac{\mu}{n+2}|h|^2.
\end{align*}
 Using (\ref{bochner3}), we therefore obtain
  \begin{align*}(\Delta_Eh,h)_{L^2}=&\left\|D_1 h\right\|^2_{L^2}+2\mu\left\| h\right\|_{L^2}^2
 -4(h,\mathring{R}h)_{L^2}\\
 \geq &2\mu\left\| h\right\|_{L^2}^2-4(h,\mathring{B}h)_{L^2}-8\frac{\mu}{n+2}\left\|h\right\|^2_{L^2}\\
 \geq &2\mu\frac{n-2}{n+2}\left\| h\right\|_{L^2}^2-4\left\|b\right\|_{L^{\infty}}\left\|h\right\|^2_{L^2}. 
 \end{align*}
 Under the assumptions of the theorem, $\Delta_{E}|_{TT}$ is nonnegative.
\end{proof}
\noindent
 We also prove a stability criterion involving the $L^{n/2}$-norm of $b$:
 \begin{thm}\label{bochnerthm2}Let $(M,g,J)$ be a positive K\"ahler-Einstein manifold\index{K\"ahler-Einstein manifold} with constant $\mu$. If the function $b$ satisfies
 \begin{align*}\left\|b\right\|_{L^{n/2}}\leq \mu\cdot\volume(M,g)^{2/n}\cdot\frac{(n-2)}{2(n+2)}\left(\frac{4(n-1)}{n(n+2)}+1\right)^{-1},
 \end{align*}
 then $(M,g)$ is stable.
 \end{thm}
 \begin{proof}The proof is very similar to that of Theorem \ref{integralweyl}. We may assume that $\volume(M,g)=1$.
 Let $h\in TT$. By assumtion, $(M,g)$ is a Yamabe metric\index{Yamabe!metric}. Thus, we can use the Sobolev inequality\index{Sobolev!inequality}  and we get
 \begin{align*}(\mathring{B}h,h)_{L^2}\leq\int_M b|h|^2\dv        &\leq\left\|b\right\|_{L^{n/2}}\left\|h\right\|_{L^{2n/n-2}}^2
                                          \leq\left\|b\right\|_{L^{n/2}}\left(\frac{4(n-1)}{\mu n(n-2)}\left\|\nabla h\right\|^2_{L^2}+\left\|h\right\|^2_{L^2}\right).
 \end{align*}
 By the above,
 \begin{align*}(\Delta_Eh,h)_{L^2}&=\left\|\nabla h\right\|^2_{L^2}-2(\mathring{R}h,h)_{L^2}\\
                               &\geq\left\|\nabla h\right\|^2_{L^2}-2(\mathring{B}h,h)_{L^2}-\frac{4\mu}{n+2}\left\|h\right\|^2_{L^2}\\
                               &\geq\left\|\nabla h\right\|^2_{L^2}-2\left\|b\right\|_{L^{n/2}}\left(\frac{4(n-1)}{\mu n(n-2)}\left\|\nabla h\right\|^2_{L^2}+\left\|h\right\|^2_{L^2}\right)-\frac{4\mu}{n+2}\left\|h\right\|^2_{L^2}\\
                               &=\left(1-\frac{8(n-1)}{\mu n(n-2)}\left\|b\right\|_{L^{n/2}}\right)\left\|\nabla h\right\|^2_{L^2}-2\left\|b\right\|_{L^{n/2}}\left\|h\right\|^2_{L^2}-\frac{4\mu}{n+2}\left\|h\right\|^2_{L^2}.
 \end{align*}
 The first term on the right hand side is nonnegative by the assumption on $b$.
 To estimate $\left\|\nabla h\right\|_{L^2}^2$, we rewrite (\ref{bochner3}) to get
 \begin{align*}\left\|\nabla h\right\|_{L^2}^2&=\left\|D_1 h\right\|^2_{L^2}+2\mu\left\| h\right\|_{L^2}^2
 -2(h,\mathring{R}h)_{L^2}\\
 &\geq2\mu\frac{n}{n+2}\left\| h\right\|_{L^2}^2-2(h,\mathring{B}h)_{L^2}\\
 &\geq2\mu\frac{n}{n+2}\left\| h\right\|_{L^2}^2-2\left\|b\right\|_{L^{n/2}}\left(\frac{4(n-1)}{\mu n(n-2)}\left\|\nabla h\right\|^2_{L^2}+\left\|h\right\|^2_{L^2}\right).
 \end{align*}
 Thus,
 \begin{align*}\left\|\nabla h\right\|_{L^2}^2\geq2\left(\mu\frac{n}{n+2}-\left\|b\right\|_{L^{n/2}}\right)\left(1+\frac{8(n-1)}{\mu n(n-2)}\left\|b\right\|_{L^{n/2}}\right)^{-1}\left\|h\right\|_{L^2}^2.
 \end{align*}
 By combining these arguments,
 \begin{align*}(\Delta_Eh,h)_{L^2}\geq \bigg\{2&\left(\mu\frac{n}{n+2}-\left\|b\right\|_{L^{n/2}}\right)\left(1-\frac{8(n-1)}{\mu n(n-2)}\left\|b\right\|_{L^{n/2}}\right)\\
                                           &\left(1+\frac{8(n-1)}{\mu n(n-2)}\left\|b\right\|_{L^{n/2}}\right)^{-1}-2\left\|b\right\|_{L^{n/2}}-\frac{4\mu}{n+2}\bigg\}\left\|h\right\|^2_{L^2},
 \end{align*}
 and the right-hand side is nonnegative if the assumption of the theorem holds.
 \end{proof}
\begin{rem}
By the Cauchy-Schwarz inequality\index{Cauchy-Schwarz inequality}, we clearly have
\begin{align}\label{bpestimate}b(p)\leq|B|_p.
\end{align}
\end{rem}
 \begin{rem}As for the Weyl tensor, there also exist isolation results for the $L^{n/2}$-norm of the Bochner tensor, see \cite[Theorem A]{IK04}. 
 The methods are similar to those used in \cite{ItSa02} and for the constant $C_n$ appearing in formula $(24)$ of \cite{IK04}, the value $1/6$ seems to be not too far away from the optimum.
 A criterion combining Theorem \ref{bochnerthm2} and \eqref{bpestimate} is not ruled out by these results, if $n\geq5$.
 If $n=4$, $B=W^-$ (see \cite[p.~232]{IK04}). Then Theorem \ref{GLtheorem} applies and this criterion is ruled out.
 \end{rem}

\end{document}